\documentclass[
final
 , nomarks
]{dmtcs-episciences}

\usepackage[utf8]{inputenc}
\usepackage{subfigure}
\usepackage{amsmath, tikz}

\newtheorem{thr}{Theorem}[section]

\newtheorem{lem}[thr]{Lemma}
\newtheorem{prop}[thr]{Proposition}
\newtheorem{conj}[thr]{Conjecture}
\newtheorem{cor}[thr]{Corollary}

\newtheorem{defi}[thr]{Definition}

\def\T{\mathbb{T}}
\def\P{\mathcal{P}}

\DeclareMathOperator{\ecc}{ecc}

\DeclareMathOperator{\Id}{Id}
\DeclareMathOperator{\A}{A}
\DeclareMathOperator{\D}{D}
\DeclareMathOperator{\W}{W}
\DeclareMathOperator{\M}{M}
\DeclareMathOperator{\WW}{WW}
\DeclareMathOperator{\TW}{TW}
\DeclareMathOperator{\Sz}{Sz}
\DeclareMathOperator{\wSz}{wSz}

\usepackage[round]{natbib}

\author{Stijn Cambie\thanks{ Email: \href{mailto:stijn.cambie@hotmail.com}{stijn.cambie@hotmail.com}. This work has been supported by a Vidi Grant of the Netherlands Organization for Scientific Research (NWO), grant number $639.032.614$.} }%
\title{Five results on maximizing topological indices in graphs}
\affiliation{
	Radboud University Nijmegen, The Netherlands}
\received{2020-11-11}
\revised{2021-03-25, 2021-08-16}
\accepted{2021-09-29}
\begin{document}
\publicationdetails{23}{2021}{3}{10}{6896}

\tikzstyle{every node}=[circle, draw, fill=black!50,
inner sep=0pt, minimum width=4pt]

\definecolor{xdxdff}{rgb}{0.49019607843137253,0.49019607843137253,1.}
\definecolor{ududff}{rgb}{0.30196078431372547,0.30196078431372547,1.}

\tikzstyle{every node}=[circle, draw, fill=black!50,
inner sep=0pt, minimum width=5pt]
\maketitle
\begin{abstract}
In this paper, we prove a collection of results on graphical indices.
We determine the extremal graphs attaining the maximal generalized Wiener index (e.g. the hyper-Wiener index) among all graphs with given matching number or independence number. This generalizes some work of Dankelmann, as well as some work of Chung. 
We also show alternative proofs for two recent results on maximizing the Wiener index and external Wiener index by deriving it from earlier results.
We end with proving two conjectures. We prove that the maximum for the difference of the Wiener index and the eccentricity is attained by the path if the order $n$ is at least $9$ and that the maximum weighted Szeged index of graphs of given order is attained by the balanced complete bipartite graphs.
\end{abstract}

\section{Introduction}

Let $G$ be a simple connected graph, as we only work with connected graphs in this paper. We denote its vertex set by $V(G)$ and its edge set by $E(G).$
The independence number of a graph $G$, denoted by $\alpha(G)$, is the size of the largest independent vertex set. 
The matching number of a graph $G$ is the size of a maximum independent edge subset of $G$, we will denote it by $m(G)$ or $m$.
We will denote by $\T(n,m)$ the set of all trees with $n$ vertices and matching number $m$. A path $P_n$ is a path of order $n$.

Let $d(u, v)$ denote the distance between vertices $u$ and $v$ in a graph $G$.
The diameter $d(G)$ of a graph equals $\max_{u,v \in V(G)} d(u,v).$
The eccentricity of a vertex $v$, $\varepsilon(v)$ equals $\max_{u \in V(G)} d(u,v).$ The eccentricity of a graph $G$ is the sum of the eccentricities over all vertices, i.e. $\varepsilon(G)=\sum_{v \in V} \varepsilon(v).$

The degree of the vertex $u$ will be denoted $\deg(u)$.
For an edge $e=uv$, $n_u(e)$ will be equal to the number of vertices $x$ for which $d(x,u)<d(x,v)$.
The Wiener index of a graph $G$ equals the sum of distances between all unordered pairs of vertices, i.e. $$\W(G)=\sum_{\{u,v\} \subset V(G)} d(u,v).$$
The mean distance of the graph $G$ equals $\mu(G)=\frac{\W(G)}{\binom n2}.$
Some general form of mean distance can be derived from the notion of power means.

\begin{defi}
	The $j^{th}$ power mean of $n$ positive real numbers $x_1,x_2,\ldots, x_n$ is $$\M_j(x_1,\ldots, x_n)= \sqrt[j]{\frac{x_1^j+x_2^j+\ldots+x_n^j}{n}}.$$
	When $j=0$, $\M_0(x_1,\ldots, x_n)= \sqrt[n]{x_1x_2\ldots x_n}$.
	\\Furthermore $\M_{\infty}(x_1,\ldots, x_n)=\max \{x_1,x_2,\ldots, x_n\}, \M_{-\infty}(x_1,\ldots, x_n)=\min \{x_1,x_2,\ldots, x_n\}$.
\end{defi}

Other graphical indices used in this paper, are the hyper-Wiener index, the external Wiener index, terminal Wiener index, Szeged index and weighted Szeged index. They are defined respectively as 
$$\WW(G)= \frac12 \sum_{\{u,v\} \subset V(G)} d^2(u,v)+d(u,v)$$
$$\W_{ex}(G)=\sum_{u,v \in V(G), \min\{\deg(u),\deg(v)\}=1} d(u,v) $$
$$\TW(G)=\sum_{u,v \in V(G), \deg(u)=\deg(v)=1} d(u,v) $$
$$\Sz(G) = \sum_{e= \{u,v\}\in E(G)} n_u(e) \cdot n_v(e)$$
$$\wSz(G) = \sum_{e= \{u,v\}\in E(G)} \left(\deg(u)+\deg(v)\right) \cdot n_u(e) \cdot n_v(e)$$

In Section~\ref{sec:alt_dank} we give an alternative proof for a theorem of Dankelmann~\cite{D} on the maximum Wiener index of a connected graph with given order and matching number. We prove this for a notion of generalized Wiener index $\W_f$, implying the result for e.g. the hyper-Wienerindex.
Due to a relation between order, matching number and independence number, we also observe a power mean version of a result of Chung~\cite{C}.
We present alternative, short proofs for the main results of~\cite{DIS} and~\cite{JL19} based on results known before in Sections~\ref{sec:3} and~\ref{sec:4} respectively. Also we give a proof for Conjecture $4.3$ in \cite{DAKD21} in Section~\ref{sec:5} and for Conjecture $1$ in~\cite{BFJS} in Section~\ref{sec:6}.

\section{Maximum generalized Wiener index given $m$ or $\alpha$}\label{sec:alt_dank}

Theorems $2.14$, $3.10$ and $4.7$ in the survey of \cite{OverviewWW} give the extremal graphs attaining the minimum hyper-Wiener index among all graphs with given order and matching number, for the family of graphs being the connected graphs, the trees and the unicyclic graphs respectively. Some general version was proven in~\cite{CWZ17} as the result holds for a more general class of indices represented by $F.$
In this section we will prove the analog for the maximum. The general statement works for a different class of distance-based indices.

\begin{defi}
	The generalized Wiener indices are of the form $$\W_f(G)= \sum_{\{u,v\} \subset V(G)} f\left(d(u,v) \right)$$ where $f$ is a convex function satisfying $f(0)=0$ which is strictly increasing on $\mathbb{R}^+$.
\end{defi}

Note that when we take $f \equiv \Id$ or $f \colon x \mapsto \binom{x+1}{2}$, we get the Wiener index or the hyper-Wiener index, respectively.
The condition that $f(0)=0$ is just a handy convention, as $\W_f$ is just shifted with $\binom{n}{2} c$ if one shifts $f$ with a constant $c.$
The additional constraint that $f$ is convex (when comparing with the result in~\cite{CWZ17})  is added to have the same extremal graph for the whole class of indices.

Let $\A_{n,m}$ be a path with $2m-1$ vertices, with one leaf of the path connected to $ \lceil \frac{n-2m+1}{2} \rceil $ different vertices and the other leaf with $ \lfloor \frac{n-2m+1}{2} \rfloor $ pendent vertices, i.e. it is a balanced double broom with $n-(2m-1)$ leaves when $n \ge 2m+1$ and if $n=2m$, it is a path.

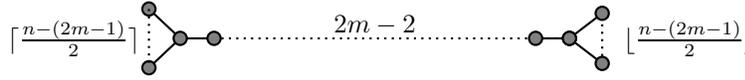
\begin{figure}[h]
	\centering
	\begin{tikzpicture}[thick,scale=0.9]%
	\draw \foreach \a in {10,-10} \foreach \x in {180} {
		(\x:2)--(\x+\a:2.5) node {}
	};
	\draw[dotted] \foreach \x in {0} {
		(-1.5,0)--(3.75,0)  
	};
	\draw 	(0:3.25)--  (0:3.75)  node {};
	\draw 	(-2,0)--  (-1.5,0)  node {};
	\draw[dotted] (170:2.5)--(190:2.5);

	\draw[dotted]  \foreach \x in {0} {
		(\x-5:4.25)--(\x+5:4.25) 
	};
	
	\draw \foreach \a in {5,-5} \foreach \x in {0} {
		(\x:3.75)--(\x+\a:4.25) node {}
	};

	\coordinate [label=center:$2m-2$] (A) at (0.875,0.2);

	\coordinate [label=right:$\lfloor \frac{n-(2m-1)}{2} \rfloor $] (A) at (0:4.5); 
	\coordinate [label=left:$\lceil \frac{n-(2m-1)}{2} \rceil $] (A) at (180:2.55);

	\draw	(3.25,0)  node {};
	\draw	(-2,0)  node {};	
	\draw	(3.75,0)  node {};

	\end{tikzpicture}
	\caption{Extremal graph $\A_{n,m}$}
	\label{fig:graphDankelmann}
\end{figure}

For any generalized Wiener index $\W_f$, we will prove that $\A_{n,m}$ is the unique extremal graph $G$ attaining the maximum value of $\W_f(G)$ among all graphs having order $n$ and matching number $m$. This was known already for the Wiener index by Dankelmann~\cite{D}.

We will use a kind of tree rearrangements, which we call subtree pruning and regrafting (SPR). It was defined in~\cite{SC19}, but for completeness we give the definition here again.

\begin{defi}[SPR]
	Let $G$ be a graph.
	Given a rooted subtree $S$ of $G$, such that the root $d=S \cap H.$
	Pruning $S$ from $G$ is removing the whole structure $S$ excluding the root $d$.
	Regrafting $S$ at a vertex $v$, means that we are taking a copy $S'$ of $S$ which we insert at $v$, letting its root $d'$ coincide with $v$. No additional edges are drawn in this process.
\end{defi}

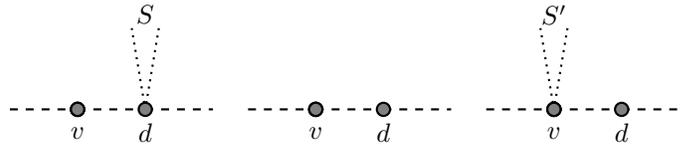
\begin{figure}[h]
	\centering
	
	\begin{tikzpicture}[thick,scale=0.9]%
	\draw[dotted] (1.2,1.2)--(1,0)--(0.8,1.2);
	\draw[dashed] (-1,0) -- (0,0);
	\draw[dashed] (0,0) -- (1,0);
	\draw[dashed] (1,0)  -- (2,0);
	\draw node{} (0,0) node{} ;
	\draw node{} (1,0) node{} ;
	\coordinate [label=center:$S$] (A) at (1,1.4);
	\coordinate [label=center:$d$] (A) at (1,-0.35);
	\coordinate [label=center:$v$] (A) at (0,-0.35);
	
	\end{tikzpicture} \quad
	\begin{tikzpicture}[thick,scale=0.9]%
	
	\draw[dashed] (-1,0) -- (0,0);
	\draw[dashed] (0,0) -- (1,0);
	\draw[dashed] (1,0)  -- (2,0);
	\draw node{} (0,0) node{} ;
	\draw node{} (1,0) node{} ;
	\coordinate [label=center:$d$] (A) at (1,-0.35);
	\coordinate [label=center:$v$] (A) at (0,-0.35);
	\end{tikzpicture} \quad
	\begin{tikzpicture}[thick,scale=0.9]%
	\draw[dotted] (0.2,1.2)--(0,0)--(-0.2,1.2);
	\draw[dashed] (-1,0) -- (0,0);
	\draw[dashed] (0,0) -- (1,0);
	\draw[dashed] (1,0)  -- (2,0);
	\draw node{} (0,0) node{} ;
	\draw node{} (1,0) node{} ;
	\coordinate [label=center:$S'$] (A) at (0,1.4);
	\coordinate [label=center:$d$] (A) at (1,-0.35);
	\coordinate [label=center:$v$] (A) at (0,-0.35);
	
	\end{tikzpicture}

	\caption{the graph $G$, $S$ being pruned from $G$ and $S$ being regrafted at $v$}
	\label{fig:SPR}
\end{figure}

We know that extremal graphs are trees, since deleting an edge which is not part of a maximum matching will increase the generalized Wiener index as at least one distance strictly increases.

We will use the notation $\W_f(\T(n,m)) = \max \{ \W_f(G) \mid G \in \T(n,m)\}$. 
For $m=1$, the extremal graphs are stars. So from now onwards, we assume $m>1$, which implies that the diameter of the extremal graph (being a tree) is at least $3$.

The first proposition we need for the proof is the following.

\begin{prop}\label{monotonicityTree}
	For fixed $n$, when $m_1<m_2$ and the sets $\T(n,m_1)$ and $\T(n,m_2)$ are both nonempty, then $\W_f(\T(n,m_1)) < \W_f(\T(n,m_2))$.
\end{prop}

\begin{proof}
	Assume this proposition is not true. In that case there exist some $n$ and $m$ such that 
	$\T(n,m)$ and $\T(n,m+1)$ are both nonempty and $\W_f(\T(n,m)) \ge  \W_f(\T(n,m+1)).$
	For some fixed $n$, we take the least integer $m$ for which $\W_f(\T(n,m)) \ge  \W_f(\T(n,m+1))$ holds and take an extremal graph $G \in \T(n,m)$ with $\W_f(G)=\W_f(\T(n,m))$.
	
	Since $G$ is a tree which is not a path (as a path reaches the largest possible matching number), we can choose a leaf $\ell$ and a vertex $w$ of degree at least $3$ such that $d(w,\ell)$ is the smallest among all such choices.
	Considering $G$ as a rooted tree in $w$, there are at least three branches, the path $P$ from $\ell$ to $w$ being one of them.
	Let $S$ be a branch different from $P$ and $S'$ be the union of the remaining branches different from $P$ and $S$. Here we do not consider $w$ as a vertex of $S$ nor of $S'$.
	We can prune the subtree $S$ (with root $w$) and regraft it at $\ell$.
	After this operation, the set of distances between $P$ and $S$ or $S'$ is the same as before, while the distance between any vertex of $S'$ and any vertex of $S$ has increased with $d(w, \ell)$.
	Since $f$ is a strictly increasing function, this implies that $\W_f$ has strictly increased by performing the SPR operation, while the matching number has not increased with more than one. This implies that $\W_f(G)=\W_f(\T(n,m)) \ge  \W_f(\T(n,i))$ for every $i \le m+1$ was not true. This contradiction implies that the proposition is true.
\end{proof}

Let $G$ be an extremal graph and $u_0$ and $u_d$ be two vertices such that the distance between them equals the diameter and the path $\P$ between them equals $u_0u_1 \ldots u_d$.
If $G$ is the path $\P$, we are in a trivial case since $P$ itself is of the form $A_{n,m}$.
In the other case, there are some subtrees attached to the path $\P$.
The following proposition gives more information about them.

\begin{prop}\label{starpathstarlike}
	There are no vertices of $\P$ different from $u_1$ and $u_{d-1}$ having degree at least $3$.
\end{prop}

\begin{proof}
	If the proposition is not true, there is an extremal graph $G$ with a subtree $S$ connected to some $u_i$ with $1<i<d-1.$
	Let $G_1$ and $G_2$ be the graphs by pruning and regrafting $S$ at $u_1$ resp. $u_{d-1}$. 
	Let $H$ be the graph $G \backslash S,$ i.e. the graph $G$ when $S$ is pruned
	Note that every neighbour of a leaf will be in a maximum matching.
	In particular, without loss of generality, we can assume $u_0u_1$ and $u_{d-1}u_d$ are edges in a maximum matching of $H, G, G_1$ or $G_2$.
	This implies that the matching number of both $G_1$ and $G_2$ is exactly equal to $m(H)+m(S),$ while the matching number of $G$, $m(G)$, is at least equal to $m(H)+m(S)$ (and plausible one larger).
	So the matching number of both $G_1$ and $G_2$ is not larger than the matching number of $G.$
	Let $S_1$ be the copy of $S$ which is connected to $u_1$ and $S_2$ be the copy of $S$ which is connected to $u_{d-1}.$
	We will prove that 
	\begin{equation}\label{afschatting}
	(d-1-i)\W_f(G_1)+(i-1)\W_f(G_2)  >(d-2)\W_f(G).
	\end{equation}
	For every $v \in H, v' \in S$ (here we take $v'$ in $S_1$ as the vertex corresponding to the original $v'$ and similarly in $S_2$) we have
	$(d-1-i) d_{G_1}(v',v)+(i-1) d_{G_2}(v',v) \ge (d-2) d_{G}(v',v)$ since 
	$(d-1-i) d_{G}(u_1,v)+(i-1) d_{G}(u_{d-1},v) \ge (d-2) d_{G}(u_i,v).$	
	Since $f$ is convex and strictly increasing, $(d-1-i) f\left( d_{G_1}(v',v) \right)+(i-1) f \left( d_{G_2}(v',v) \right) \ge (d-2) f \left( d_{G}(v',v)\right).$  From this and the fact that it is strict for $v=u_i$, Equation~(\ref{afschatting}) follows.
	Hence at least one of the two graphs $G_1$ or $G_2$ has a larger generalized Wiener index than $G$ and so taking into account Proposition~\ref{monotonicityTree} we conclude $G$ was not extremal.\end{proof}

\begin{thr}\label{thrDankGen}
	Let $G$ be a graph of order $n$ with matching number $m.$
	When $f$ is a strictly increasing, convex function, then $\W_f(G) \le \W_f(\A_{n,m})$ with equality if and only if $G \cong \A_{n,m}$.
\end{thr}

\begin{proof}
	As a consequence of Proposition~\ref{starpathstarlike}, the extremal graph is a path $u_1u_2u_3 \ldots u_{d-1}$ with $a$ pendent vertices to $u_1$ and $b$ pendent vertices to $u_{d-1}$, where $a,b \ge 1$.
	The maximum of the generalized Wiener index $\W_f$ for such a graph with matching number equals $m$ clearly needs a diameter being equal to $2m$ if $n \ge 2m+1$ and is the path if $n=2m.$ Since 
	$$\W_f(G)=\W_f(P_{2m-1})+(a+b) \sum_{i=1}^{d-1} f(i) + \frac{(a+b)(a+b-1)}{2} f(2) + ab \left( f(d)-f(2) \right)$$
	with $f$ strictly increasing, i.e. $f(d)-f(2)>0$, the maximum occurs when $\lvert a-b \rvert \le 1$ as we are considering the nontrivial case with $d \ge 3$ (as $m \ge 2$) and $a+b=n-(2m-1)$ being fixed. 
\end{proof}

As an immediate corollary, we determine the extremal graphs attaining the maximum generalized Wiener index among all graphs having order $n$ and independence number $\alpha$ for some regime.

\begin{thr}
	If $G$ is a connected graph with independence number $n-1 \ge \alpha \ge \frac{n}{2}$, then $\W_f(G)\le \W_f(\A_{n,n-\alpha})$, with equality if and only if $G=\A_{n,n-\alpha}.$
\end{thr}

\begin{proof}
	Note that for any graph, the sum $\alpha + m \le n$, since given an independent set $I$ and a matching $M$, any edge of $M$ contains at least one vertex which is not in $I.$
	This implies that $m \le n - \alpha \le \frac n2.$
	Applying Proposition~\ref{monotonicityTree} (recall extremal graphs with respect to $m$ are trees) and Theorem~\ref{thrDankGen}, we have that $\W_f(G) \le \W_f(\A_{n,n-\alpha})$. Since the graph $\A_{n,n-\alpha}$ has independence number $\alpha$, it is the unique extremal graph.
\end{proof}

In the case $2\le \alpha < \frac n2$, the proof of Dankelmann~\cite{D} can be extended to $\W_f$, as well. In that case the extremal graph being the balanced dumbbell graph $\D_{n, \alpha}$ of diameter $2 \alpha -1.$ This is the graph obtained when connecting two vertices from two cliques of almost equal order $\lceil \frac n2 \rceil -\alpha +2$ and $ \lfloor \frac n2 \rfloor -\alpha +2$ by a path of length $2\alpha-3$.

\begin{figure}[h]
	\centering
	\begin{tikzpicture}[thick,scale=0.9]%
	\draw \foreach \a in {0} \foreach \x in {180} {
		(\x:2)--(\x+\a:2.5) node {}
	};
	\draw[dotted] \foreach \x in {0} {
		(-1.5,0)--(3.75,0)  
	};
	\draw 	(0:3.25)--  (0:3.75)  node {};
	\draw 	(-2,0)--  (-1.5,0)  node {};

	\draw \foreach \a in {0} \foreach \x in {0} {
		(\x:3.75)--(\x+\a:4.25) node {}
	};

	\coordinate [label=center:$2\alpha-3$] (A) at (0.875,0.2);

	\coordinate [label=right:$K_{\lfloor \frac{n-(2\alpha-2)}{2} \rfloor }$] (A) at (0:4.5); 
	\coordinate [label=left:$K_{\lceil \frac{n-(2\alpha-2)}{2} \rceil }$] (A) at (180:2.6);

	\draw	(3.25,0)  node {};
	\draw	(-2,0)  node {};	
	\draw	(3.75,0)  node {};
	\draw [fill=xdxdff] (-2.5,0.) circle (4.5pt);	
	\draw [fill=xdxdff] (4.25,0.) circle (4.5pt);	
	\end{tikzpicture}
	\caption{Extremal graph $\D_{n,\alpha}$}
	\label{fig:graphDankelmann_alpha}
\end{figure}
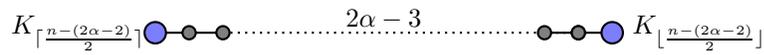

\begin{thr}
	If $G$ is a connected graph with independence number $2\le  \alpha < \frac{n}{2}$, then $\W_f(G)\le \W_f(\D_{n,\alpha})$, with equality if and only if $G=\D_{n,\alpha}.$
\end{thr}

As corollaries, we get power mean versions of the result of Chung~\cite{C}, which states that the average distance is bounded by the independence number.

\begin{thr}\label{alg_Chung}
	Let $\mu_j(G)$ be the $j^{th}$ power mean of the distances $\left\{d(u,v)\right\}_{\{u,v\}\subset V(G)}$. 
	Then for $j\ge 1$ and any connected graph $G$, one has $\mu_j(G) \le \M_j\left(2\alpha(G)-1,1\right).$
\end{thr}

\begin{proof}
	The function $f \colon x \mapsto x^j $ is a convex function on $\mathbb R^+$ when $j \ge 1$.
	Note that for this $f$, we have $\mu_j(G) = \sqrt[j]{\frac{\W_f{G}}{\binom n2}}.$
	So from the previous two theorems, we know the maximum is attained when $G$ equals $\A_{n,n- \alpha}$ or $\D_{n, \alpha}.$
	Note that for every $1 \le i \le \alpha -1$, there are more pairs of vertices $\{u,v\}$ with $d(u,v)=i$ then pairs with $d(u,v)=2\alpha -i$ in the extremal graph. Combining with $\M_j(1,2\alpha-1) \ge \M_j(i, 2\alpha-i)$ (true by the inequality of Jensen), we conclude. 
\end{proof}

By the observation that $(2\alpha-1)^j+1 \le (2 \alpha)^j$ when $j \ge 1$, we also have the following corollary. Note that for $j=1$, it is exactly the original result of Chung.

\begin{cor}\label{genChung}
	For every $j \ge 1$ and graph $G$, we have $\mu_j(G)\le 2^{ \frac{j-1}{j}} \alpha{(G)}.$ Equality holds if and only if $j=\alpha(G)=1$.
\end{cor}

\section{Maximum external Wiener index of graphs}\label{sec:3}

In this section, we give a short alternative proof for the main result in~\cite{DIS}, which proves Conjecture $11$ in~\cite{GFD16}. We show that the conjecture is basically a corollary of a theorem in~\cite{GFP}.

Remark that if $T$ is a spanning tree of a graph $G$, then $\W_{ex}(G)\le \W_{ex}(T)$ since $\deg_T(u) \le \deg_G(u)$ and $d_G(u,v) \le d_T(u,v)$.

Note that for any two vertices $u$ and $v$ in a tree, there are two leaves such that the path between them goes through $u$ and $v$.
This implies that adding an edge between two non-neighbours of any tree will strictly decrease the external Wiener index as at least one term got smaller (or even vanishes). As a consequence, any extremal graph is a tree.

Hence the result will follow from the following lemma, as we know the extremal graphs are trees.

\begin{lem}\label{keylem}
	Let $T$ be a tree of order $n$ with $\ell$ leaves.
	Then 
	\begin{align*}
	\sum_{u,v \in V(G): \deg(u)>1, \deg(v)=1} d(v,u) &\le \ell \frac{(n-\ell)(n-\ell+1)}{2},\\
	\TW(T) &\le \ell(\ell-1)+ \left \lfloor \frac{\ell^2}{4} \right \rfloor (n-\ell -1).
	\end{align*}
\end{lem}

\begin{proof}
	Let $U=\{u \in V(G), \deg(u)>1\}$ be the sets of nonleafs, which has size $n-\ell.$
	Note that  for every leaf $v$, $G[U \cup \{v\}]$ is a connected graph and hence $$\sum_{U} d(v,u) \le \sum_{i=1}^{n-\ell} i = \frac{(n-\ell)(n-\ell+1)}{2}.$$
	Equality holds if and only if $G[U]$ is a path and $v$ is connected to an endvertex of $G[U]$.
	The second part is Theorem 4 of \cite{GFP}, with the addition that it is also true for $\ell \in \{2,3\}.$
\end{proof}

\begin{thr}
	The graphs on $n$ vertices with the maximum external Wiener index $\W_{ex}$ are balanced double brooms.
\end{thr}

\begin{proof}
	Assume $T$ is the extremal graph and it has $\ell$ leaves.
	Note that $$\W_{ex}(T) = \sum_{u,v \in V(G): \deg(u)>1, \deg(v)=1} d(v,u) + \TW(T)$$ is bounded by $$\ell \frac{(n-\ell)(n-\ell+1)}{2}+ \ell(\ell-1)+ \left \lfloor \frac{\ell^2}{4} \right \rfloor (n-\ell -1)$$ due to Lemma~\ref{keylem}. Equality in the first part holds if and only if $T$ is a double broom. A double broom for which equality holds in the second equality need to be balanced and balanced double brooms attain equality.
	The maximum among all graphs is now attained by the double brooms having $\ell$ leaves, where $2 \le \ell\le n-1$ is an integer maximizing the expression.
\end{proof}

\section{Maximum Wiener index of unicyclic graphs with given bipartition}\label{sec:4}

In this section, we show that the answer to Problem $11.6$ in~\cite{KST16} is mainly a corollary of the proof of Problem $11.4$ in the same survey, which was adressed in~\cite{SC19}.
The problem, being the unsolved part in~\cite{Du12}, was recently solved in~\cite{JL19} and~\cite{BJM19}. Nevertheless, one can observe that the proof by deducing it from earlier work is much shorter.

We start with proving a lemma dealing with the case that the cycle is not minimal.

\begin{lem}\label{lem:maxWuniC2k}
	Among all unicyclic graphs of order $n \ge 2k$ containing an even cycle $C_{2k}$, the Wiener index is maximized by the graph formed by attaching a path of order $n-2k$ to a vertex of a $C_{2k}$.
\end{lem}

\begin{proof}
	We can prove this by induction, the $n=2k$ case being the trivial base case as $C_{2k}$ is the only unicyclic graph on $2k$ vertices containing a $C_{2k}.$
	Assume the lemma is true for $n-1 \ge 2k$.
	Any unicyclic graph $G$ of order $n>2k$ containing a $C_{2k}$ has at least one leaf $v$. Let $H= G \backslash v$.
	Note that the distance from $v$ to the $C_{2k}$ is at most $n-2k$ and the diameter of $G$ is at most $n-k$.
	At least $k-1$ consecutive distances between $v$ and vertices of $C_{2k}$ appear twice, these are at most $n-2k+1$ up to $n-k-1$.
	Thus $\sum_{u \in H} d(v,u) \le \sum_{i=1}^{n-k} i + \sum_{i=n-2k+1}^{n-k-1} i$ and together with the induction hypothesis on $H$, we get the result.	
\end{proof}

Using the notation as has been done in~\cite{SC19} (Figure $9$), the theorem is stated below.

\begin{thr}\label{thr:bip}
	The maximum Wiener index among all $n$-vertex unicyclic graphs
	with bipartition sizes $p,q$ ($1<p \le q$) is attained by exactly one graph, $G^4_{  \lceil{ \frac{q-p}2 \rceil} ,\lfloor{ \frac{q-p}2 \rfloor}, 2p-4  }.$
\end{thr}

\begin{proof}
	Note that a graph $G$ having bipartition of sizes $q \ge p$ has a matching number $m$ which is at most $p$.
	If $q \ge p+3$, the result is a consequence of Theorem 7.1 and Proposition 2.1 from~\cite{SC19} applied on $n=p+q$ and $m=p$.
	If $q \in \{p,p+1\}$, we have to take the maximum over all possible graphs containing an even cycle. 
	
	The maximum for the graphs in Lemma~\ref{lem:maxWuniC2k} is attained when $k=2$. 
	There are multiple ways to see this. 
	Let $G$ be $C_{2k}$ with a path $P_{n-2k}$ attached to it.
	Let $v$ be a neighbour of the vertex with degree $3$, or a random vertex of $G=C_{2k}$ if $n=2k.$
	Then $$\sum_{u,w \in G\backslash v} d(u,w) \le W(P_{n-1}) \mbox{ and } 
	\sum_{u \in G\backslash v} d(u,v) \le 2\cdot 1 + 2 \cdot 2+ 3+ \ldots + (n-3).$$
	Equality holds if and only if $k=2.$
	
	If $q=p+2$, we know the extremal graph containing a $C_4$ with $n=2p+2$ and $m\le p$ is $G^4_{n/2-m,n/2-m,2m-4}$ by Section $7$ in~\cite{SC19}.
	Furthermore $W(G^4_{n/2-m,n/2-m,2m-4}) < W(G^4_{1,1,2p-4})$ if $m<p$.
	
	The graph $G^4_{1,1,2p-4}$ has a larger Wiener index than the extremal graphs in Lemma~\ref{lem:maxWuniC2k} for $k\ge 3$, from which we conclude again.
	For this it is enough to note that for all $k\ge3$, we have 
	$$W(G^4_{1,1,2k-6})=\frac 43 k^3 - \frac{19}3 k +11 > k^3 = W(C_{2k}),$$
	as adding the path of order $n-2k$ to both structures only makes the difference larger.	 
\end{proof}

\section{Maximum difference of Wiener Index and Eccentricity}\label{sec:5}

In this section, we prove Conjecture $4.3$ in \cite{DAKD21}.

\begin{thr}\label{thr:W-eps}
	For $n \ge 9$, among all graphs with order $n$, $W(G)-\varepsilon(G)$ is maximized by $P_n.$ Moreover, $P_n$ is the unique extremal graph.
\end{thr}

To start with, we prove that we can focus on trees as deleting an edge does not decrease the quantity $(W-\varepsilon)$, where $(W-\varepsilon)(G)$ denotes $W(G)-\varepsilon(G)$.

\begin{lem}\label{lem:trees_extremal}
	Let $G$ be a graph with order $n \ge 9$ and radius at least $3$. Let $e$ be an edge such that $G\backslash e$ is connected.
	Then $(W-\varepsilon)(G)\le (W-\varepsilon)(G \backslash e).$
\end{lem}

\begin{proof}
	Let $e=uv$ and assume $(W-\varepsilon)(G)> (W-\varepsilon)(G \backslash e).$
	Suppose the shortest cycle in $G$ containing $e$ is $C_k$.
	Note that distances do not decrease when deleting edges, so we have to focus on the eccentricities that increase.
	Let $z$ be a vertex not belonging to $C_k$ for which the eccentricity increases when deleting $e$.
	Without loss of generality we can assume $d(z,u)<d(z,v)$. Let $\ecc_{G\backslash e}(z)=d_{G\backslash e}(z,t)$ for a vertex $t$ (where possible $t=v$).
	Then we know that $d(z,t)<d_{G\backslash e}(z,t)$, so the shortest path from $z$ to $t$ in $G$ uses the edge $uv$ and so $d(z,t)=d(z,v)+d(v,t).$
	This also implies that $d(v,t)=d_{G\backslash e}(v,e).$
	Combining these observations with the definition of eccentricity and the triangle inequality, we get that
	\begin{align*}
	\ecc_{G\backslash e}(z) - \ecc_{G}(z) &\le d_{G\backslash e}(z,t)-d(z,t)\\&\le d_{G\backslash e}(z,v)+ d_{G\backslash e}(v,t) - ( d(z,v) + d(v,t))\\&=  d_{G\backslash e}(z,v)-d(z,v).	
	\end{align*}
	As the difference in eccentricity for $z$ is cancelled by the difference of distance between $z$ and $v$, while $z$ was taken arbitrary,  $(W-\varepsilon)(G)>(W-\varepsilon)(G \backslash e)$
	implies that 
	\begin{equation}\label{eq:C_kvsP_k}
	(W-\varepsilon)(C_k)>(W-\varepsilon)(P_k) \Leftrightarrow \frac{k}{2} \left \lfloor \frac{k^2}{4} \right \rfloor -k \left \lfloor \frac{k}{2} \right \rfloor > \binom{k+1}{3} - \left \lfloor \frac 34 k^2 - \frac k2 \right \rfloor
	\end{equation} 
	which is only the case if $k \in \{3,5\}$ and then the difference is exactly equal to $1$.
	
	\textit{Case: $k=3$} To have a contradiction, both $\ecc(u)$ and $\ecc(v)$ need to increase when deleting $e$. This implies there is a vertex $w$ such that $3 \le \ecc(u)=d(w,u)<d_{G \backslash e}(w)$. Let $v_2$ be the neighbour of $v$ on a shortest path in $G$ from $v$ to $w.$
	Then $d_{G \backslash e}(v_2,u)=d(v_2,u)+1.$ 
	So if $\ecc_{G \backslash e}(v_2)=\ecc(v_2)$, we have an additional difference that (in the expansion of $W(G\backslash e)-W(G)$) is at least $1$, leading to a contradiction.
	If $\ecc_{G \backslash e}(v_2)>\ecc(v_2)\ge 3$, there is an other vertex $x$ not belonging to the $C_3$ such that $d_{G \backslash e}(x,v_2)>d(x,v_2)$ and we conclude again.
	
	\textit{Case: $k=5$} 
	In this case, let $v_2$ be the neighbour from $v$ in the $C_5$ different from $u$.
	When doing the check in the reduced case that $(W-\varepsilon)(C_k)>(W-\varepsilon)(P_k)$, we take into account that the eccentricity of $v_2$ goes up by $1$ in $C_5$.
	So $(W-\varepsilon)(G)>(W-\varepsilon)(G \backslash e)$ is only possible if $\ecc_{G \backslash e}(v_2)>\ecc(v_2).$
	But since $\ecc(v_2) \ge 3$, there is a vertex $x$ not belonging to the $C_5$ for which $d_{G \backslash e}(v_2,x)>d(v_2,x).$
	As we did not take this difference into account before when looking to $(W-\varepsilon)(G)-(W-\varepsilon)(G \backslash e)$, we conclude that $(W-\varepsilon)(G)\le (W-\varepsilon)(G \backslash e)$ again.	
\end{proof}

Having proven this lemma, it is essentially enough to prove it for trees, once checking the result for graphs with radius at least $2.$ For this a bit of work is needed (as one can expect due to the behaviour of the extremal graphs for $n \le 8$).

\begin{proof}[Proof of Theorem~\ref{thr:W-eps}]
	First, one can check that $W(G)-\varepsilon(G)$ is smaller than $(W-\varepsilon)(P_n)$ when $15 \ge n \ge 9$ in case $G$ has radius at most $2$.
	For $n=9$, a brute force check confirms.
	For $n=10$, if the diameter is at most $3$, then $W(G) \le 97$ and $\varepsilon(G) \ge 10$ and we are done. If the diameter is $4$, we have $W(G) \le 117$ and $\varepsilon(G) \ge 2\cdot 4 +2\cdot 3+6\cdot 2=26$, so $W(G)-\varepsilon(G) \le 91 < 95=(W-\varepsilon)(P_{10}).$
	The cases $11 \le n \le 15 $ work similarly.
	For $n \ge 16$, note that the diameter of $G$ is bounded by $4$ and so $W(G) \le 4\binom{n}{2}<(W-\varepsilon)(P_n).$
	
	So from now on, we only have to consider graphs with radius at least $3$ and hence by Lemma~\ref{lem:trees_extremal} we can first focus on trees.
	A bruteforce check by computer over all trees of order $9 \le n \le 20$ verifies the theorem for these values.
	For $n\ge 21$, we first observe that the diameter of an extremal graph is at least $7$, since otherwise 
	$(W-\varepsilon)(G) \le 6 \binom{n}{2}-3n<(W-\varepsilon)(P_n).$
	Now we can prove the statement by induction.
	If the diameter $d$ equals $n-1$, we have the path $P_n$.
	If the diameter is smaller, $d \le n-2$, one can remove a leaf $v$ which is not part of the diameter.
	Now the eccentricities of all vertices different from $v$ are the same in $G$ and $G \backslash v$.
	By the induction hypothesis, we have $(W-\varepsilon)(G \backslash v) \le (W-\varepsilon)(P_{n-1}).$
	Furthermore we have 
	\begin{align*}
	(W-\varepsilon)(G)-(W-\varepsilon)(G \backslash v)
	&=\sum_{u \in G \backslash v} d(u,v) - \ecc(v)\\
	&<1+2+\ldots +(n-2) \\
	&=(W-\varepsilon)(P_n)-(W-\varepsilon)(P_{n-1})
	\end{align*} 
	from which we conclude.	This implies that $P_n$ is the unique tree maximizing $(W-\varepsilon)(G).$
	By Lemma~\ref{lem:trees_extremal} no graph with a spanning tree which is not a path can be extremal (note that the radius does not decrease when removing edges).	
	Since the cycle $C_n$ (and the path $P_n$ itself) is the only graph which has only $P_n$ a spanning graph, but $(W-\varepsilon)(C_n)<(W-\varepsilon)(P_n)$ for $n>5$, as concluded from the computation in Equation~\ref{eq:C_kvsP_k}, the path $P_n$ is the unique extremal graph.
\end{proof}

Additionally, we add two remarks on the work in~\cite{DAKD21} about the minimum for $W-\varepsilon.$
In their Theorem $3.2$, the authors prove that the minimum for $(W-\varepsilon)(T)$ among trees is attained by caterpillars.
More precisely, the extremal tree is $P_n$ when $n\le 6
$ and if $n\ge7$, the extremal tree is attained for the star $S_{n-1}$ with one edge subdivided.
This is mainly a corollary of the following lemma, as it implies that we only have to compare a few possible trees.
\begin{lem}
	Among all trees with fixed diameter $d$ and order $n$, the minimum of $(W-\varepsilon)$ occurs if and only if we have a path $P_{d+1}$ of diameter $d$ with the $n-d-1$ remaining vertices connected to the same vertex on the path, which is a central vertex if $d$ is odd, or a central vertex or neighbour of it, if $d$ is even.
\end{lem}
\begin{proof}
	Let $U$ be the set of vertices different from the ones on the diameter $v_0v_1\ldots v_d$. Then $\sum_{u,v \in U} d(u,v) \le 2\binom{n-d-1}{2}$ with equality if and only if they are all connected to the same vertex on the diameter.
	We note that for every $u \in U$, $\sum_{0 \le i \le d} d(u,v_i) - \varepsilon(v)$ is minimal if $u$ is connected to a central vertex and equality is also possible if it is connected to a neighbouring vertex of a central vertex when $d+1$ is odd as then there is only one central vertex.
	Here we use that $d(v_j,v_i)+d(v_j,v_{d-i})\ge d-2i$ where $u_j$ is the neighbour of $u$ and $0 \le i < \frac d2$. 
\end{proof}

We now also determine this minimum among all graphs.

\begin{prop}
	For every graph $G$ of order $n$, we have $(W-\varepsilon)(G)\ge \left \lceil \frac{n(n-4)}2\right \rceil.$
\end{prop}
\begin{proof}
	For every vertex $v$, one has 
	$\sum_{u \in V \backslash v} d(u,v) -2\varepsilon(v) \ge n-4$,
	since all distances are at least equal to one, there is a vertex $u$ with $d(u,v)=\varepsilon(v)$ and an other one with $d(w,v)\ge \varepsilon(v)-1$.
	Summing over all vertices $v$, we conclude that $2(W-\varepsilon)(G) \ge n(n-4).$
	Dividing by $2$ and observing that $(W-\varepsilon)(G)$ is always an integer, we conclude.
\end{proof}
The minimum of $(W-\varepsilon)(G)$ is attained by the complement of 
\[\begin{cases}
\frac{n}{2} K_2 &\text{if } n \text{ is even,}\\
\frac{n-1}{2} K_2\cup K_1 \text{ or } \frac{n-3}{2} K_2\cup P_3 & \text{if } n \text{ is odd.}
\end{cases}
\]
For $n \ge 6$, this is the exact characterization of the extremal graphs.
The graph $P_4$ for $n=4$ and the one sketched in Figure~\ref{fig:extrG_n5} for $n=5$ are also extremal.

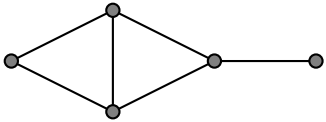
\begin{figure}[h]
	\centering
	\begin{tikzpicture}[thick,scale=0.9]%

	\draw 	(-1.5,0)--  (0,-0.75)  ;
	\draw 	(-1.5,0)--  (0,0.75)  ;
	\draw 	(1.5,0)--  (0,-0.75)  ;
	\draw 	(1.5,0)--  (0,0.75)  ;	
	\draw (0,-0.75)--(0,0.75);
	\draw (1.5,0)--(3,0);

	\draw	(-1.5,0)  node {};	
	\draw	(0,0.75)  node {};
	\draw	(1.5,0)  node {};	
	\draw	(0,-0.75)  node {};
	\draw	(3,0)  node {};
	\end{tikzpicture}
	\caption{Additional extremal graph for $n=5$}
	\label{fig:extrG_n5}
\end{figure}

\section{Maximum weighted Szeged index}\label{sec:6}
In this section, we prove the following open conjecture, posed in~\cite{BFJS}.

\begin{conj}[Conjecture 1 in~\cite{BFJS}]\label{conj:BFJS}
	For any $n$-vertex graph $G$, the weighted Szeged index of $G$, $\wSz(G)$, is upper-bounded by $\wSz( K_{\lfloor \frac n2 \rfloor, \lceil \frac n2 \rceil})$ and equality is only attained by
	the balanced complete bipartite graph $K_{\lfloor \frac n2 \rfloor, \lceil \frac n2 \rceil}$, or $K_3$ if $n=3.$
\end{conj}

First, we make the following crucial observation.
\begin{lem}\label{lem1}
	For any edge $e=uv \in E(G)$, we have $n_u(e)+n_v(e)+\deg(u)+\deg(v) \le 2n.$
\end{lem}

\begin{proof}
	Since $n_u(e)+n_v(e) \le n$, this is trivial when $\deg(u)+\deg(v) \le n$. If $\deg(u)+\deg(v) > n$, then $u$ and $v$ have at least $\deg(u)+\deg(v) - n$ neighbours $x$ in common, which satisfy $d(x,u)=d(x,v)$ and hence do not belong to $N_u(e)$ nor $N_v(e).$
	Hence $n_u(e)+n_v(e) \le n-(\deg(u)+\deg(v) - n)$ which is equivalent with $n_u(e)+n_v(e)+\deg(u)+\deg(v) \le 2n.$
\end{proof}

Let the degrees of the vertices of the graph be $a_1,a_2, \ldots, a_n$. For an edge $e=uv$ whose end vertices have degree $a_i$ and $a_j$, let 
$x_e=\frac{ a_i+ a_j}{2}$ be the average degree of the two endvertices.

Then by double counting, we find the following two equalities (the first one being the hand shaking lemma)
\begin{lem}\label{lem:doublecount}
	We have 
	$$
	\sum_i a_i = 2 |E| \mbox{ and }
	\sum_i a_i^2 =2 \sum_e x_e.
	$$
\end{lem}
Next, we prove two propositions which are the main ingredients for the proof.

\begin{prop}\label{red_wSz}
	For every edge $e=uv$, we have 
	$$\left(\deg(u)+\deg(v)\right) \cdot n_u(e) \cdot n_v(e) \le 2  \left \lfloor \frac {n^2}4 \right \rfloor  (n-x_e).$$
\end{prop}

\begin{proof}
	Combining $n_u(e) \cdot n_v(e) \le \left( \frac{n_u(e) + n_v(e)}2\right)^2$ (by AM-GM) and Lemma~\ref{lem1}, we have $$\left(\deg(u)+\deg(v)\right) \cdot n_u(e) \cdot n_v(e) \le 2x_e( n-x_e)^2.$$
	Now we have $x_e(n-x_e) \le  \left \lfloor \frac {n^2}4 \right \rfloor $ if $n$ is even or when $n$ is odd and $x_e \not= \frac{n}{2}$ (as $2x_e$ is integral).
	
	When $x_e=\frac n2$, then Lemma~\ref{lem1} gives $n_u(e) + n_v(e) \le n$ and hence $n_u(e) \cdot n_v(e) \le  \lfloor \frac {n^2}4 \rfloor$ and the conclusion holds again.
\end{proof}

\begin{prop}\label{sum_wSz}
	For any graph $G$, we have
	$$\sum_e 2\left(n-x_e \right) \le n \left \lfloor \frac {n^2}4 \right \rfloor .$$
\end{prop}

\begin{proof}
	Using Lemma~\ref{lem:doublecount}, we get 
	\begin{align*}
	\sum_e 2\left(n-x_e \right) &=2|E| n -2 \sum_e x_e\\
	&= \sum_i a_i (n- a_i)\\
	&\le n   \lfloor \frac {n^2}4 \rfloor.
	\end{align*}
\end{proof}

\begin{proof}[Proof of Conjecture~\ref{conj:BFJS}]
	Combining Proposition~\ref{red_wSz} and Proposition~\ref{sum_wSz}, we find that the weighted Szeged index of the graph satisfies
	
	\begin{align*}
	\wSz(G) &= \sum_e \left(\deg(u)+\deg(v)\right) \cdot n_u(e) \cdot n_v(e)\\
	&\le    \left \lfloor \frac {n^2}4 \right \rfloor  \sum_e 2\left(n-x_e \right) \\
	&\le n  \left( \left \lfloor \frac {n^2}4 \right \rfloor  \right)^2\\
	&=\wSz( K_{\lfloor \frac n2 \rfloor, \lceil \frac n2 \rceil})
	\end{align*}
	
	If $n$ is even, equality holds if and only there is equality in every step. In particular $a_i=\frac n2$ for all $i$ and thus $n_u(e)=n_v(e)=\frac n2$ for every edge $e=uv$, which also implies that the graph should be triangle-free as well since $n_u(e)+n_v(e)=n$ for every edge $e.$
	But then $u$ and $v$ are connected with disjoint independent sets of order $\frac n2$ and since the degree of every vertex is exactly $\frac n2$, we conclude $G \cong K_{ \frac n2 , \frac n2 }$.
	
	If $n=3$, we see that $\wSz(K_3)=\wSz(K_{2,1})$ and these two graphs are the only extremal ones.
	
	If $n\ge 5$ is odd, equality holds if and only if $G \cong K_{\lfloor \frac n2 \rfloor, \lceil \frac n2 \rceil}$.
	Note that all $a_i$ need to be equal to $\lfloor \frac n2 \rfloor$ or  $\lceil \frac n2 \rceil$ to have equality in Proposition~\ref{sum_wSz}.
	Note that we also need equality in Lemma~\ref{lem1} for every edge to have equality in Proposition~\ref{red_wSz}.
	So it is impossible that $x_e < \frac n2$ and if $x_e>\frac n2$, we have a triangle with three vertices which all need to have degree $\lceil \frac n2 \rceil$ and do not create other triangles, which is impossible. Hence the conclusion follows as $x_e = \frac n2$ for every edge, i.e. the end vertices of any edge have degree $\lfloor \frac n2 \rfloor$ and $ \lceil \frac n2 \rceil .$
\end{proof}

\acknowledgements
\label{sec:ack}
The author is very grateful to the anonymous referees for the time they dedicated to this article and for the comments they made which improved this manuscript.

\nocite{*}
\bibliographystyle{abbrvnat}
\bibliography{genDankelmann}
\label{sec:biblio}

\end{document}